\RequirePackage{snapshot}
\documentclass[leqno]{siamltex}

\usepackage{siampaper}
\usepackage{braket}
\usepackage{enumerate} %

\newcommand{\polar}[1]{#1^\circ}

\DeclareMathOperator{\vspan}{span}

\title{%
  Gauge optimization and duality%
  \thanks{Research partially supported by NSERC Discovery grant 312104, and
    NSERC Collaborative Research and Development grant 375142-08.}%
}
\author{
  Michael P. Friedlander%
  \thanks{%
    Department of Computer Science,
    University of British Columbia,
    Vancouver, BC, Canada.
    E-mail: \texttt{\{mpf,ijamj,tkpong\}@cs.ubc.ca}.
    MPF was also
    supported as a visiting scholar at the Simons Institute on
    Theoretical Computer Science at UC Berkeley; TKP was also supported
    as a PIMS Postdoctoral Fellow.
    \hfill\bf October 9, 2013; revised March 19, 2014
  }
  \and
  Ives Mac\^edo$^{\dagger}$\!\!%
  \and
  Ting Kei Pong$^{\dagger}$%
}

\begin{document}

\maketitle

\thispagestyle{plain}
\pagestyle{myheadings}

\begin{abstract}
  Gauge functions significantly generalize the notion of a norm, and
  gauge optimization, as defined by \cite{freund:1987}, seeks the
  element of a convex set that is minimal with respect to a gauge
  function. This conceptually simple problem can be used to model a
  remarkable array of useful problems, including a special case of
  conic optimization, and related problems that arise in machine
  learning and signal processing. The gauge structure of these
  problems allows for a special kind of duality framework. This paper
  explores the duality framework proposed by \citeauthor{freund:1987},
  and proposes a particular form of the problem that exposes some
  useful properties of the gauge optimization framework (such as the
  variational properties of its value function), and yet maintains most
  of the generality of the abstract form of gauge optimization.
\end{abstract}

\section{Introduction}

One approach to solving linear inverse problems is to optimize a
regularization function over the set of admissible deviations between
the observations and the forward model. Although regularization
functions come in a wide range of forms depending on the particular
application, they often share some common properties.  The aim of this
paper is to describe the class of gauge optimization problems, which
neatly captures a wide variety of regularization formulations that
arise in fields such as machine learning and signal processing. We
explore the duality and variational properties particular to this
family of problems.

All of the problems that we consider can be expressed as
\begin{equation*}
  \label{g-primal}\tag{P}
  \minimize{x\in\Xscr} \quad \kappa(x) \textt{subject to} x\in\Cscr,
\end{equation*}
where $\Xscr$ is a finite-dimensional Euclidean space,
$\Cscr\subseteq\Xscr$ is a closed convex set, and
$\kappa:\Xscr\to\R\cup\{+\infty\}$ is a \emph{gauge} function, i.e., a
nonnegative, positively homogeneous convex function that vanishes at
the origin. (We assume that $0\notin \Cscr$, since otherwise the
origin is trivially a solution of the problem.) This class of problems
admits a duality relationship that is different from Lagrange duality,
and is founded on the gauge structure of its objective. Indeed,
\cite{freund:1987} defines the dual counterpart
\begin{equation*}
  \label{g-dual}\tag{D}
  \minimize{y\in\Xscr} \quad \kappa^\circ(y) \textt{subject to} y\in\Cscr',
\end{equation*}
where the set
\begin{equation}\label{eq:1}
  \Cscr'=\set{y|\innerp{y}{x}\ge 1\text{\ for all\ }x\in\Cscr}
\end{equation}
is the antipolar of $\Cscr$ (in contrast to the better-known polar of a convex
set), and the polar $\polar\kappa$ (also a gauge) is the function that best
satisfies the Cauchy-Schwartz-like inequality most tightly:
\begin{equation}\label{eq:4}
  \innerp{x}{y} \le \kappa(x)\,\kappa^{\circ}(y),
  \qquad \forall x\in\dom\kappa,\ \forall y\in\dom\kappa^{\circ};
\end{equation}
see \eqref{eq:polardef} for the precise definition.
It follows directly from this inequality and the definition of
$\Cscr'$ that all primal-dual feasible pairs $(x,y)$ satisfy the
weak-duality relationship
\begin{equation}\label{eq:9}
  1 \le \kappa(x)\,\polar\kappa(y),
  \qquad
  \forall x\in\Cscr\cap \dom\kappa,\ \forall y\in\Cscr'\cap\dom\kappa^\circ.
\end{equation}
This duality relationship stands in contrast to the more usual
Lagrange framework, where the primal and dual objective values bound
each other in an additive sense.

\subsection{A roadmap}

Freund's analysis of gauge duality is mainly concerned with
specialized linear and quadratic problems that fit into the gauge
framework, and with the pair of abstract problems \eqref{g-primal} and
\eqref{g-dual}. Our treatment in this paper considers the particular
formulation of
\eqref{g-primal} given by
\begin{equation*}
  \label{eq:7}
  \tag{P\!$_{\smash\rho}$}
  \minimize{x\in \Xscr} \quad \kappa(x) \quad\st\quad \rho(b-Ax)\le\sigma,
\end{equation*}
where $\rho$ is also a gauge. Typical applications might use $\rho$ to
measure the mismatch between the model $Ax$ and the measurements $b$,
and in that case, it is natural to assume that $\rho$ vanishes only at
the origin, so that the constraint reduces to $Ax=b$ when $\sigma=0$. This formulation
is only very slightly less general than \eqref{g-primal} because any
closed convex set can be represented as $\set{x|\rho(b-x)\le1}$ for
some vector $b$ and gauge $\rho$; cf.~\S\ref{sec:gauges}. However, it
is sufficiently concrete that it allows us to develop a calculus for
computing gauge duals for a wide range of existing problems. (Conic
side constraints and a linear map in the objective can be easily
accommodated; this is covered in~\S\ref{sec:extensions}.)

The special structure of the functions in the gauge
program~\eqref{eq:7} leads to a duality framework that is analogous to
the classical Lagrange-duality framework.  The gauge dual program
of~\eqref{eq:7} is
\begin{equation*}
  \label{g-dual2}
  \tag{D\!$_{\smash\rho}$}
  \minimize{ y\in \Xscr}\quad \polar\kappa(A^*y)
  \quad\st\quad
  \innerp y b - \sigma\polar\rho(y) \ge 1,
\end{equation*}
which bears a striking similarity to the Lagrange dual problem
\begin{equation*}
  \label{l-dual}
  \tag{D$_{\smash\ell}$}
  \maximize{ y\in \Xscr}\quad
  \langle y,b\rangle-\sigma\rho^\circ( y)
  \quad\st\quad
  \kappa^\circ(A^* y)\leq1.
\end{equation*}
Note that the objective and constraints between the two duals play
different roles. (These two duals are derived in \S\ref{sec:duality}
under suitable assumptions.)  A significant practical difference
between these two formulations is when $\rho$ is a simple Euclidean
norm and $\kappa$ is a more complicated function (such as one
described by Example~\ref{sec1:ex2} below). The result is that the
Lagrange dual optimizes a ``simple'' objective function over a
potentially ``complicated'' constraint; in contrast, the situation is
reversed in the gauge optimization formulation.

We develop in \S\ref{sec:antipolar} an antipolar calculus for
computing the antipolars of sets such as
$\set{x|\rho(b-Ax)\le\sigma}$, which corresponds to the constraint in
our canonical formulation~\eqref{eq:7}. This calculus is applied in
\S\ref{sec:duality} to derive the gauge dual~\eqref{g-dual2}.

The formal properties of the polar and antipolar operations are
described in~\S\S\ref{sec:background}--\ref{sec:antipolar}.
In~\S\ref{sec:strong-duality} we develop conditions sufficient for
strong duality, i.e., for there to exist a primal-dual pair that
satisfies~\eqref{eq:9} with equality. Our derivation parts with the
``ray-like'' assumption used by Freund, and in certain cases further
relaxes the required assumptions by leveraging connections with
established results from Fenchel duality.

\subsection{Examples}

The following examples illustrate the versatility of the gauge
optimization formulation.
\begin{example}[Norms and minimum-length solutions]
  Norms are special cases of gauge functions that are finite
  everywhere, symmetric, and zero only at the origin. (Semi-norms drop
  the last requirement, and allow the function to be zero at other
  points.) Let $\kappa(x) = \norm{x}$ be any norm, and
  $\Cscr=\set{x|Ax=b}$ describe the solutions to an underdetermined
  linear system. Then~\eqref{g-primal} yields a minimum-length
  solution to the linear system $Ax=b$. This problem can be modeled as
  an instance of \eqref{eq:7} by letting $\rho$ be any gauge
  function for which $\rho^{-1}(0)=\set0$ and setting $\sigma=0$. The polar
  $\polar\kappa = \norm{\cdot}\subd$ is the norm dual to
  $\norm{\cdot}$, and $\Cscr'=\set{A^{*} y | \innerp{b}{y}\ge 1}$; cf.
  Corollary~\ref{cor:rho-axb-antipolar}. The corresponding gauge
  dual~\eqref{g-dual} is then
  \[
  \minim_{y\in\Xscr}\quad\norm{A^*y}\subd\quad\st\quad\innerp{b}{y}\ge1.
  \]
  \vspace{-\baselineskip}
\end{example}

\begin{example}[Sparse optimization and atomic norms]\label{sec1:ex2}
  In his thesis, \cite{Berg:2009} describes a framework for sparse
  optimization based on the formulation where $\kappa$ is a gauge, and
  the function $\rho$ is differentiable away from the origin. The
  nonnegative regularization parameter $\sigma$ influences the degree
  to which the linear model $Ax$ fits the observations $b$. This
  formulation is specialized by \cite{BergFriedlander:2011} to the
  particular case in which $\rho$ is the 2-norm. In that case,
  $\Cscr=\set{x|\norm{Ax-b}_{2}\le\sigma}$ and
  \[
  \Cscr'=\set{A^{*} y | \innerp{b}{y} - \sigma\norm{y}_{2}\ge1};
  \]
  cf. Corollary~\ref{cor:rho-antipolar}. \cite{TeuS:2013} consider a
  related case where the misfit between the model and the observations
  is measured by the Kullback-Leibler divergence.

  \cite{Chan:2012} describe how to construct regularizers that
  generalize the notion of sparsity in linear inverse problems. In
  particular, they define the gauge
  \begin{equation}\label{eq:3}
  \norm{x}_{\Ascr} := \inf\set{ \lambda\ge0 \mid x\in\lambda\conv\Ascr}
  \end{equation}
  over the convex hull of the set of canonical atoms given by the set $\Ascr$.
  If $0\in \interior\conv\Ascr$ and $\Ascr$ is bounded and symmetric,
  i.e., $\Ascr=-\Ascr$, then the definition~\eqref{eq:3} yields a
  norm. For example, if $\Ascr$ consists of the set of unit
  $n$-vectors that contain a single nonzero element, then~\eqref{eq:3}
  is the 1-norm; if $\Ascr$ consists of the set of rank-1 matrices
  with unit spectral norm, then~\eqref{eq:3} is the Schatten
  1-norm. The polar $\polar\kappa(y) =
  \sup\set{\innerp{y}{a}|a\in\conv(\{0\}\cup\Ascr)}$ is the support
  function of the closure of
  $\conv(\{0\}\cup\Ascr)$. \cite{jaggi:2013} catalogs various sets of
  atoms that yield commonly used gauges in machine learning.
\end{example}

\begin{example}[Conic gauge optimization] \label{example:conic-gauge}
  In this example we demonstrate that it is possible to cast any
  convex conic optimization problem in the gauge framework. Let
  $\Kscr$ be a closed convex cone, and let $\Kscr^{*}$ denote its
  dual. Consider the primal-dual pair of feasible conic problems:
  \begin{subequations}
  \begin{alignat}{4}
    \label{eq:2}
      &\minimize{x} \quad&\innerp{c}{x} &\quad\st&\quad Ax = b,\ x&\in\Kscr,
    \\&\maximize{y} \quad&\innerp{b}{y} &\quad\st&\quad   c-A^*y  &\in\Kscr^*.
  \end{alignat}
  \end{subequations}
  Suppose that $\yhat$ is a dual-feasible point, and define
  $\chat=c-A^*\yhat$. Because $\chat\in\Kscr^*$, it follows that
  $\innerp{\chat}{x}\ge0$ for all $x\in\Kscr$. In particular, the
  primal problem can be equivalently formulated as a gauge
  optimization problem by defining
  \begin{equation}\label{eq:8}
    \kappa(x) = \innerp{\chat}{x} + \indicator{\Kscr}(x)
    \textt{and}
    \Cscr=\set{x \mid Ax=b},
  \end{equation}
  where $\delta_{\Kscr}$ is the indicator function on the set
  $\Kscr$. (More generally, it is evident that any function of the
  form $\gamma + \delta_{\Kscr}$ is a gauge if $\gamma$ is a gauge.)
  This formulation is a generalization of the nonnegative linear
  program discussed by Freund, and we refer to it as conic gauge
  optimization. The generalization captures some important problem
  classes, such as trace minimization of positive semidefinite (PSD)
  matrices, which arises in the phase-retrieval problem
  \citep{candes2012phaselift}. This is an example where $c\in\Kscr^*$,
  in which case the dual-feasible point $\yhat=0$ is trivially
  available for the gauge reformulation; cf.~\S\ref{sec:SDP-gauge}.

  A concrete example of the simple case where $c\in\Kscr^*$ is the
  semidefinite programming relaxation of the max-cut problem studied by
  \cite{Goemans:1995}. Let $G=(V,E)$ be an undirected graph, and
  $D=\diag\big((d_v)_{v\in V}\big)$, where $d_v$ denotes the degree of
  vertex $v\in V.$ The max-cut problem can be formulated as
  \begin{equation*}
    \label{eq:maxcut}
    \maximize{x} \quad \tfrac{1}{4}\innerp{D-A}{xx^{T}}
    \quad\st\quad
    x\in\{-1,1\}^V,
  \end{equation*}
  where $A$ denotes the adjacency matrix associated with $G.$ The
  semidefinite programming relaxation for this problem is derived by
  ``lifting'' $xx\T$ into a PSD matrix:
  \begin{equation*}
    \label{eq:maxcut:sdp:max}
    \maximize{X} \quad \tfrac{1}{4}\innerp{D-A}{X}
    \quad\st\quad \diag X=e,\ X\succeq0,
  \end{equation*}
  where $e$ is the vector of all ones, and $X\succeq 0$ denotes the
  the PSD constraint on $X$.
  The constraint $\diag X=e$ implies that $\innerp{D}{X}=\sum_{v\in
    V}d_v=2|E|$ is constant. Thus, the optimal value is equal to
  \begin{equation}
    \label{eq:maxcut:sdp:min}
    |E| - \tfrac{1}{4}\cdot
    \min_{X}\Set{\innerp{D+A}{X}|\diag X=e,\ X\succeq0},
  \end{equation}
  and the solution can be obtained by solving this latter problem.
  Note that $D+A$ is PSD because it has nonnegative
  diagonals and is diagonally dominant. (In fact, it is possible to
  reduce the problem in linear time to one where $D+A$ is positive definite
  by identifying its bipartite connected components.)
  Because the dual of the cone of PSD matrices is itself,
  and the trace inner product between PSD matrices is nonnegative,
  \eqref{eq:maxcut:sdp:min} falls into the class of conic gauge
  problems defined by~\eqref{eq:2}.
\end{example}

\begin{example}[Submodular functions]
  Let $V=\set{1,\ldots,n}$, and consider the set-function
  $f:2^{V}\to\R$, where $f(\emptyset)=0$. The \cite{Lovasz:1983}
  extension $\fhat:\R^{n}\to\R$ of $f$ is given by
  \[
  \fhat(x) = \sum_{k=1}^{n}x_{j_{k}}
  \big[
  f(\set{j_{1},\ldots,j_{k}}) - f(\set{j_{1},\ldots,j\km1})
  \big],
  \]
  where $x_{j_{1}}\ge x_{j_{2}}\ge\cdots\ge x_{j_{n}}$ are the sorted
  elements of $x$. Clearly, the extension is positively homogeneous
  and vanishes at the origin. As shown by Lov\'asz, the extension is
  convex if and only if $f$ is \emph{submodular}, i.e.,
  \[
  f(A) + f(B) \ge f(A\cup B) + f(A\cap B) \textt{for all} A,B\subset V;
  \]
  see also \cite[Proposition~2.3]{bach2011learning}. If $f$ is
  additionally non-decreasing, i.e.,
  \[
  A, B \subset V\text{ and }  A\subset B \quad\Longrightarrow\quad f(A)\le f(B),
  \]
  then the extension is nonnegative over $\R^n_+$. Thus, when $f$ is a
  submodular and non-decreasing set function, that function plus the
  indicator on the nonnegative orthant, i.e., $\fhat +
  \delta_{\R^n_+}$, is a gauge. \cite{bach2011learning} surveys the
  properties of submodular functions and their application in machine
  learning; see Proposition~3.7 therein.
\end{example}

\section{Background and notation} \label{sec:background}

In this section we review known facts about polar sets, gauges and
their polars, and introduce results that are useful for our subsequent
analysis. We mainly follow \citet{Roc70}: see~\S14 in that text for a
discussion of polarity operations on convex sets, and~\S15 for a
discussion of gauge functions and their corresponding polarity operations.

We use the following notation throughout. For a closed convex set
$\Dscr$, let $\Dscr_\infty$ denote the recession cone of $\Dscr$
\citep[Definition 2.1.2]{AuT03}, and $\ri\Dscr$ and $\cl \Dscr$ denote,
respectively, the relative interior and the closure of $\Dscr$.  The
indicator function of the set $\Dscr$ is denoted by $\delta_\Dscr$.

For a gauge $\kappa:\Xscr\to\R\cup\set\infty$, its domain is denoted
by $\dom \kappa = \set{x|\kappa(x)<\infty}$, and its epigraph is
denoted by $\epi\kappa = \set{(x,\mu)|\kappa(x)\le \mu}$.  A function
is called closed if its epigraph is closed, which is equivalent to the
function being lower semi-continuous \citep[Theorem~7.1]{Roc70}.  Let
$\cl \kappa$ denote the gauge whose epigraph is $\cl\epi\kappa$, which
is the largest lower semi-continuous function smaller than $\kappa$
\citep[p.~52]{Roc70}. Finally, for any $x\in \dom \kappa$, the
subdifferential of $\kappa$ at $x$ is denoted $\partial \kappa(x) =
\set{y|\kappa(u) -\kappa(x)\ge \innerp{y}{u-x},\ \forall u}$.

We make the following blanket assumptions throughout. The set $\Cscr$
is a nonempty closed convex set that does not contain the origin;
the set $\Dscr$ is a nonempty convex set that may or may not
contain the origin, depending on the context. The gauge function
$\rho:\Xscr\to\R\cup\set\infty$, used in \eqref{eq:7}, is closed; when
$\sigma=0$, we additionally assume that $\rho^{-1}(0)=\set0$.

\subsection{Polar sets}\label{sec:polar}

The polar of a nonempty closed convex set $\Dscr$ is defined as
\begin{equation*}\label{sec3:polardef}
  \polar\Dscr := \set{y | \innerp{x}{y}\le 1, \ \forall x\in \Dscr},
\end{equation*}
which is necessarily closed convex, and contains the origin. The bipolar theorem states that if $\Dscr$ is closed, then it contains the origin if and
only if $\Dscr = \Dscr^{\circ\circ}$ \citep[Theorem~14.5]{Roc70}.

When $\Dscr = \Kscr$ is a closed convex cone, the polar is
equivalently given by
\begin{equation*}\label{sec3:polardef_cone}
  \polar\Kscr := \set{y |  \innerp{x}{y}\le 0, \ \forall x\in \Kscr}.
\end{equation*}
The positive polar cone (also known as the dual cone) of $\Dscr$ is
given by
\begin{equation*}\label{sec3:dualdef_cone}
  \Dscr^* := \set{y | \innerp{x}{y}\ge 0, \ \forall x\in \Dscr}.
\end{equation*}
The polar and positive polar are related via the closure of the conic
hull, i.e.,
\[
  \Dscr^*
  =  (\cl\cone\Dscr)^*
  = -(\cl\cone\Dscr)^\circ,
  \textt{where}
  \cone\Dscr := \bigcup_{\lambda\ge 0}\lambda \Dscr.
\]

\subsection{Gauge functions} \label{sec:gauges}

All gauges can be represented in the form of a Minkowski
function $\gamma_{\Dscr}$ of some nonempty convex set $\Dscr$, i.e.,
\begin{equation}\label{eq:11}
  \kappa(x) = \gamma_\Dscr(x)
  := \inf\set{\lambda\ge 0 | x\in \lambda \Dscr}.
\end{equation}
In particular, one can always choose $\Dscr=\set{x | \kappa(x)\le 1}$,
and the above representation holds.  The polar of the
gauge $\kappa$ is defined by
\begin{equation}\label{eq:polardef}
  \polar\kappa(y):=
  \inf\set{\mu>0 | \innerp{x}{y}\le \mu \kappa(x),\ \forall x},
\end{equation}
which explains the inequality~\eqref{eq:4}. Because $\kappa$ is a proper
convex function, one can also define its convex conjugate:
\begin{equation} \label{eq:10}
  \kappa^*(y) := \sup_x\set{\innerp{x}{y} - \kappa(x)}.
\end{equation}
It is well known that $\kappa^*$ is a proper closed convex function
\cite[Theorem~12.2]{Roc70}.  The following proposition collects properties
that relate the polar and conjugate of a gauge.

\begin{proposition}\label{sec2:prop1}
  For the gauge $\kappa:\Xscr\to\R\cup\set\infty$, it holds that
  \begin{enumerate}[{\rm(i)}]
  \item $\kappa^\circ$ is a closed gauge function;
  \item $\kappa^{\circ\circ} = \cl\kappa = \kappa^{**}$;
  \item $\kappa^\circ(y) = \sup_x\set{\innerp{x}{y}|\kappa(x)\le 1}$ for all $y$;
  \item $\kappa^*(y) = \delta_{\kappa^\circ(\cdot)\le 1}(y)$ for all $y$;
  \item $\dom\polar\kappa = \Xscr$ if $\kappa$ is closed and
    $\kappa^{-1}(0)= \set{0}$;
  \item $\epi\kappa^\circ = \set{(y,\lambda)|(y,-\lambda)\in (\epi\kappa)^\circ}$.
  \end{enumerate}
\end{proposition}

\begin{proof}
  The first two items are proved in Theorems~15.1 and~12.2 of
  \citet{Roc70}. Item (iii) follows directly from the
  definition~\eqref{eq:polardef} of the polar gauge.
  To prove item (iv), we note that if $g(t)=t$, $t\in \mathbb{R}$, then the so-called monotone conjugate $g^+$ is
  \[
    g^+(s) = \sup_{t\ge 0}\set{st - t} = \delta_{[0,1]}(s),
  \]
  where $s\ge 0$. Now, apply \citet[Theorem~15.3]{Roc70} with $g(t)=t$, and
  $\kappa^{**}$ in place of $f$ in that theorem to obtain that $\kappa^{***}(y) = \delta_{[0,1]}(\kappa^{**\circ}(y))$. The conclusion in item (iv)
  now follows by noting that
  $\kappa^{***} = \kappa^*$ and $ \kappa^{**\circ} = \kappa^{\circ\circ\circ} =
  \kappa^{\circ}$. To prove item (v), note that the assumptions
  together with \citet[Proposition~3.1.3]{AuT03} show that $0\in
  \interior\dom\kappa^*$.  This together with item (iv) and the
  positive homogeneity of $\kappa^\circ$ shows that $\dom\kappa^\circ =
  \Xscr$. Finally, item (vi) is stated on \citet[p.~137]{Roc70} and can also be verified directly from the definition.
\end{proof}

In many interesting applications, the objective in \eqref{g-primal} is the
composition $\kappa\circ A$, where $\kappa$ is a gauge and $A$ is a linear
map. Clearly, $\kappa\circ A$ is also a gauge. The next result gives the polar
of this composition.

\begin{proposition}\label{sec2:prop1.5}
  Let $A$ be a linear map.
  Suppose that either
  \begin{enumerate}[{\rm (i)}]
    \item $\epi\kappa$ is polyhedral; or
    \item $\ri \dom\kappa\cap \range A\neq \emptyset$.
  \end{enumerate}
  Then
  \begin{equation*}\label{sec2:formula}
    (\kappa\circ A)^\circ(y) =
      \inf\limits_u\set{\kappa^\circ(u)|A^*u=y}.
  \end{equation*}
  Moreover, the infimum is attained when the value is finite.
\end{proposition}

\begin{proof}
  Since $\kappa\circ A$ is a gauge, we have from Proposition~\ref{sec2:prop1}(iii) that
  \begin{equation*}
      (\kappa\circ A)^\circ(y) = \sup_x\set{\innerp{y}{x}|\kappa(Ax)\le 1} = -\inf_x\set{\innerp{-y}{x} + \delta_\Dscr(Ax)},
  \end{equation*}
  where $\Dscr = \set{x|\kappa(x)\le 1}$.  Since $\kappa$ is
  positively homogeneous, we have $\dom \kappa =
  \bigcup_{\lambda\ge0}\lambda\Dscr$.  Hence, $\ri\dom\kappa =
  \bigcup_{\lambda > 0}\lambda \ri \Dscr$ from
  \citet[p.~50]{Roc70}. Thus, assumption (ii) implies that
  $\ri\Dscr\cap \range A\neq \emptyset$. On the other hand, assumption
  (i) implies that $\Dscr$ is polyhedral; and $\Dscr\cap \Range A\neq
  \emptyset$ because they both contain the origin.  Use these
  conclusions and apply \citet[Corollary~31.2.1]{Roc70} (see also
  Rockafellar's remark right after that corollary for the case when
  $\Dscr$ is polyhedral) to conclude that
  \begin{equation*}
  \begin{split}
      (\kappa\circ A)^\circ(y) & = -\sup_u\set{-(\innerp{-y}{\cdot})^*(-A^*u) - (\delta_\Dscr)^*(u)}\\
      & = -\sup_u\set{-\kappa^\circ(u)|A^*u=y},
    \end{split}
  \end{equation*}
  where the second equality follows from the definition of conjugate
  functions and Proposition~\ref{sec2:prop1}(iii).  Moreover, from
  that same corollary, the supremum is attained when finite. (Note
  that Rockafeller's statement of that corollary is formulated for the
  difference between convex and concave function, and must be
  appropriately adapted to our case.)  This completes the proof.
\end{proof}

Suppose that a gauge is given as the Minkowski function of a nonempty
convex set that may not necessarily contain the origin. The following
proposition summarizes some properties concerning this representation.

\begin{proposition}\label{sec2:prop2}
  Suppose that $\Dscr$ is a nonempty convex set. Then
  \begin{enumerate}[{\rm (i)}]
  \item $(\gamma_\Dscr)^\circ = \gamma_{\Dscr^\circ}$;
  \item $\gamma_\Dscr = \gamma_{{\rm conv}(\{0\}\cup\Dscr)}$;
  \item $\gamma_\Dscr$ is closed if ${\rm conv}(\{0\}\cup\Dscr)$ is closed.
  \item If $\kappa=\gamma_\Dscr$, $\Dscr$ is closed, and $0\in \Dscr$, then
    $\Dscr$ is the unique closed convex set containing the origin such that $\kappa
    = \gamma_\Dscr$; indeed, $\Dscr = \set{x|\kappa(x)\le 1}$.
  \end{enumerate}
\end{proposition}

\begin{proof}
  Item (i) is proved in \citet[Theorem~15.1]{Roc70}. Item (ii) follows
  directly from the definition. To prove (iii), we first notice from
  item (ii) that we may assume without loss of generality that $\Dscr$
  contains the origin. Notice also that $\gamma_\Dscr$ is closed if
  and only if $\gamma_\Dscr = \gamma_\Dscr^{**}$. Moreover,
  $\gamma_\Dscr^{**} =\gamma_{\Dscr^{\circ\circ}} =
  \gamma_{\cl\Dscr}$, where the first equality follows from
  Proposition~\ref{sec2:prop1}(ii) and item (i), while the second
  equality follows from the bipolar theorem.  Thus, $\gamma_\Dscr$ is
  closed if and only if $\gamma_\Dscr = \gamma_{\cl\Dscr}$.  The
  latter holds when $\Dscr = \cl\Dscr$. Finally, the conclusion in item (iv)
  was stated on \citet[p.~128]{Roc70}; indeed, the relation $\Dscr = \set{x|\kappa(x)\le 1}$
  can be verified directly from definition.
\end{proof}

From Proposition~\ref{sec2:prop1}(iv) and Proposition~\ref{sec2:prop2}(iv), it is not hard to prove
the following formula on the polar of the sum of two gauges of independent variables.

\begin{proposition}\label{sec2:prop2.5}
  Let $\kappa_1$ and $\kappa_2$ be gauges. Then
  $\kappa(x_1,x_2):= \kappa_1(x_1) + \kappa_2(x_2)$
  is a gauge, and its polar is given by
  \[
  \kappa^\circ(y_1,y_2)
  = \max\set{\kappa_1^\circ(y_1),\ \kappa_2^\circ(y_2)}.
  \]
\end{proposition}

\begin{proof}
  It is clear that $\kappa$ is a gauge. Moreover,
    \begin{equation*}
    \begin{split}
      \kappa^*(y_1,y_2)& = \kappa_1^*(y_1) + \kappa_2^*(y_2)= \delta_{\Dscr_1\times \Dscr_2}(y_1,y_2),
    \end{split}
  \end{equation*}
  where $\Dscr_i = \set{x|\kappa^\circ_i(x)\le 1}$ for $i = 1,2$; the first equality follows from the definition of the convex conjugate and the fact that $y_1$ and $y_2$ are decoupled, and
  the second equality follows from Proposition~\ref{sec2:prop1}(iv). This together with Proposition~\ref{sec2:prop2}(iv) implies that
  \begin{equation*}
    \begin{split}
     \kappa^\circ(y_1,y_2) & = \inf\set{\lambda\ge 0| y_1\in \lambda \Dscr_1,\ y_2\in \lambda \Dscr_2}\\
      & = \max\set{\inf\set{\lambda\ge 0| y_1\in \lambda \Dscr_1},\ \inf\set{\lambda\ge 0| y_2\in \lambda \Dscr_2}}\\
      & = \max\set{\gamma_{\Dscr_1}(y_1),\ \gamma_{\Dscr_2}(y_2)}=
      \max\set{\kappa_1^\circ(y_1),\ \kappa_2^\circ(y_2)}.
    \end{split}
  \end{equation*}
  This completes the proof.
\end{proof}

The following corollary is immediate from
Propositions~\ref{sec2:prop1.5} and~\ref{sec2:prop2.5}.

\begin{corollary}
  Let $\kappa_1$ and $\kappa_2$ be gauges.  Suppose that either
  \begin{enumerate}[{\rm (i)}]
    \item $\epi\kappa_1$ and $\epi\kappa_2$ are polyhedral; or
    \item $\ri \dom\kappa_1\cap \ri \dom \kappa_2\neq \emptyset$.
  \end{enumerate}
  Then
  \begin{equation}
    (\kappa_1 + \kappa_2)^\circ(y) =
      \inf_{u_1,u_2}\set{\max\set{\kappa_1^\circ(u_1),\ \kappa_2^\circ(u_2)}|u_1+u_2=y}.
  \end{equation}
  Moreover, the infimum is attained when finite.
\end{corollary}
\begin{proof}
  Apply Proposition~\ref{sec2:prop1.5} with $Ax = (x,x)$ and the gauge
  $\kappa_1(x_1) + \kappa_2(x_2)$, whose polar is given by
  Proposition~\ref{sec2:prop2.5}.
\end{proof}

The support function for a nonempty convex set $\Dscr$ is defined as
\[
\sigma_\Dscr(y)=\sup_{x\in \Dscr}\, \innerp{x}{y}.
\]
It is straightforward to check that if $\Dscr$ contains the origin, then the
support function is a (closed) gauge function. Indeed, we have the following
relationship between support and Minkowski functions
\citep[Corollary~15.1.2]{Roc70}.

\begin{proposition}\label{sec2:prop3}
  Let $\Dscr$ be a closed convex set that contains the origin. Then
  $\gamma^\circ_\Dscr = \sigma_\Dscr$ and $\sigma^\circ_\Dscr =
  \gamma_\Dscr$.
\end{proposition}

\subsection{Antipolar sets}

The antipolar $\Cscr'$, defined by~\eqref{eq:1}, is nonempty as a
consequence of the separation theorem. Freund's \citeyear{freund:1987} derivations are largely based on the following definition of a ray-like set. (As Freund mentions, the terms \emph{antipolar} and \emph{ray-like} are not universally
used.)

\begin{definition}
  A set $\Dscr$ is \emph{ray-like} if for any $x,y\in \Dscr$,
  \[
    x + \alpha y\in \Dscr \textt{for all} \alpha\ge 0.
  \]
\end{definition}
Note that the antipolar $\Cscr'$ of a (not
necessarily ray-like) set $\Cscr$ must be ray-like.

The following result is analogous to the bipolar theorem for antipolar operations; see \citet[p.~176]{McLinden:1978} and \citet[Lemma~3]{freund:1987}.

\begin{theorem}[Bi-antipolar theorem] \label{thm:bi-antipolar}
  $\Cscr = \Cscr''$ if and only if $\Cscr$ is ray-like.
\end{theorem}

The following proposition, stated by \citet[p.~176]{McLinden:1978},
follows from the bi-antipolar theorem.

\begin{proposition}\label{sec3:C''}
  $ \Cscr'' = \bigcup_{\lambda\ge 1}\lambda \Cscr.$
\end{proposition}

The next lemma relates the positive polar of a convex set, its antipolar and the recession cone of its antipolar.

\begin{lemma}\label{sec4:lem1}
  $\cl\cone(\Cscr') = \Cscr^* = (\Cscr')_\infty$.
\end{lemma}

\begin{proof}
  It is evident that $\cl\cone(\Cscr')\subseteq \Cscr^*$. To show
  the converse inclusion, take any $x\in \Cscr^*$ and fix an $x_0\in
  \Cscr'$.  Then for any $\tau > 0$, we have
  \begin{equation*}
    \innerp{c}{x + \tau x_0} \ge \tau\innerp{c}{x_0}\ge \tau
    \textt{for all}
    c\in \Cscr,
  \end{equation*}
  which shows that $x + \tau x_0\in \cone\Cscr'$. Taking the limit as $\tau$
  goes to $0$ shows that $x\in \cl\cone(\Cscr')$. This proves the first equality.

  Next we show the second equality, and begin with the observation
  that $\Cscr^*\subseteq (\Cscr')_\infty$. Conversely, suppose that
  $x\in (\Cscr')_\infty$ and fix any $x_0\in \Cscr'$. Then, by
  \citet[Proposition~2.1.5]{AuT03}, $x_0 + \tau x\in \Cscr'$ for all
  $\tau> 0$. Hence, for any $c\in \Cscr$,
  \begin{equation*}
    \frac{1}{\tau}\innerp{c}{x_0} + \innerp{c}{x} = \frac{1}{\tau}\innerp{c}{x_0+\tau x} \ge \frac{1}{\tau}.
  \end{equation*}
  Since this is true for all $\tau>0$, we must have $\innerp{c}{x}\ge 0$. Since $c\in \Cscr$ is arbitrary, we conclude that $x\in \Cscr^*$.
\end{proof}

\section{Antipolar calculus}
\label{sec:antipolar}

In general, it may not always be easy to obtain an explicit formula
for the Minkowski function of a given closed convex set
$\Dscr$. Hence, we derive some elements of an antipolar calculus that
allows us to express the antipolar of a more complicated set in terms
of the antipolars of its constituents.  These rules are useful for
writing down the explicit gauge duals of problems such
as~\eqref{eq:7}. Table~\ref{tab:antipolar-calculus} summarizes the
main elements of the calculus.

\begin{table}
  \centering
  \caption{The main rules of the antipolar calculus; the required assumptions are made explicit in the specific references. \label{tab:antipolar-calculus}}
  \begin{tabular}{ll}
  \toprule
  Result & Reference
  \\ \midrule
      $(A\Cscr)' = (A^*)^{-1}\Cscr'$ & Proposition~\ref{sec3:antipolarAC}
    \\$(A^{-1}\Cscr)' = \cl(A^*\Cscr')$ & Propositions~\ref{sec3:cor1} and~\ref{sec3:cor2}
    \\$(\Cscr_1\cup\Cscr_2)' = \Cscr_1' \cap \Cscr_2'$
                                     & Proposition~\ref{sec3:union}
    \\$(\Cscr_1\cap\Cscr_2)' = \cl\conv(\Cscr_1' \cup \Cscr_2')$
                                     & Proposition~\ref{sec3:intersection}
  \\ \bottomrule
  \end{tabular}
\end{table}

As a first step, the following formula gives an expression for the
antipolar of a set defined via a gauge. The formula follows
directly from the definition of polar functions.

\begin{proposition}\label{sec3:prop2}
  Let $\Cscr = \set{x | \rho(b - x)\le \sigma}$ with $0 <
  \sigma<\rho(b)$. Then
  \[
  \Cscr' = \set{y | \innerp{b}{y} - \sigma\rho^\circ(y)\ge 1}.
  \]
\end{proposition}

\begin{proof}
  Note that $y\in \Cscr'$ is equivalent to $\innerp{x}{y} \ge 1$ for
  all $x\in \Cscr$. Thus, for all $x$ such that $\rho(b-x)\le\sigma$,
  \begin{equation*}
  \innerp{x-b}{y}\ge 1 - \innerp{b}{y}
    \ \Longleftrightarrow\
    \innerp{b-x}{y}\le \innerp{b}{y} - 1.
  \end{equation*}
  From Proposition~\ref{sec2:prop1}(iii), this is further equivalent to $\sigma\rho^\circ(y)\le \innerp{b}{y} - 1$. %
\end{proof}

Proposition~\ref{sec3:prop2} is very general since any closed convex
set $\Dscr$ containing the origin can be represented in the form of
$\set{x | \rho(x)\le 1}$, where $\rho(x)=\inf\set{\lambda\ge 0 | x\in
  \lambda\Dscr}$; cf.~\eqref{eq:11}.
  For conic constraints in particular, one obtains the following corollary by
  setting $\rho(x) = \delta_{-\Kscr}(x)$.

\begin{corollary}\label{sec3:prop2immedcor}
  Let $\Cscr = \set{x | x\in b + \Kscr}$ for some closed convex cone
  $\Kscr$ and a vector $b\notin -\Kscr$. Then
  \[
  \Cscr' = \set{y\in \Kscr^* | \innerp{b}{y} \ge 1}.
  \]
\end{corollary}

Note that Proposition~\ref{sec3:prop2} excludes the potentially
important case $\sigma=0$; however, Corollary~\ref{sec3:prop2immedcor}
can instead be applied by defining $\Kscr = \rho^{-1}(0) =
\set{0}$. %

\subsection{Linear transformations}

We now consider the antipolar of the image of $\Cscr$ under a linear
map $A$.

\begin{proposition}\label{sec3:antipolarAC}
  It holds that
  \[
  (A\Cscr)' = (A^*)^{-1}\Cscr'.
  \]
  Furthermore, if $\cl(A\Cscr)$ does not contain the origin, then both
  sets above are nonempty.
\end{proposition}

\begin{proof}
  Note that $y\in (A\Cscr)'$ is equivalent to
  \begin{equation*}
      \innerp{y}{Ac}=\innerp{A^* y}{c} \ge 1 \textt{for all} c\in \Cscr.
  \end{equation*}
  The last relation is equivalent to $A^*y \in \Cscr'$. Hence,
  $(A\Cscr)' = (A^*)^{-1}\Cscr'$.  Furthermore, the assumption that $\cl(A\Cscr)$ does not contain the origin,
  together with an argument using supporting hyperplanes, implies $(A\Cscr)'$
  is nonempty. This completes the proof.
\end{proof}

We have the following result concerning the pre-image of
$\Cscr$.

\begin{proposition}\label{sec3:cor1}
  Suppose that $A^{-1}\Cscr\neq \emptyset$. Then
  \[
  (A^{-1}\Cscr)' = \cl(A^*\Cscr'),
  \]
  and both sets are nonempty.
\end{proposition}

\begin{proof}
  Recall from our blanket assumption (cf. \S\ref{sec:background}) that $\Cscr$ is a closed convex set not containing the origin.
  It follows that $\cl(A^*\Cscr')$ is nonempty. Moreover, $A^{-1}\Cscr$ is also a
  closed convex set that does not contain the origin. Hence,
  $(A^{-1}\Cscr)'$ is also nonempty.

  We next show that $\cl(A^*\Cscr')$ does not contain the
  origin. Suppose that $y\in A^*\Cscr'$ so that $y = A^*u$ for some
  $u\in \Cscr'$. Then for any $x\in A^{-1}\Cscr$, we have $Ax\in
  \Cscr$ and thus
  \[
  \innerp{x}{y} = \innerp{x}{A^*u} = \innerp{Ax}{u} \ge 1,
  \]
  which shows that $y\in (A^{-1}\Cscr)'$. Thus, we have
  $A^*\Cscr'\subseteq (A^{-1}\Cscr)'$ and consequently that
    $\cl(A^*\Cscr)\subseteq (A^{-1}\Cscr)'$.  Since the set $A^{-1}\Cscr$ is nonempty,
  $(A^{-1}\Cscr)'$ does not contain the origin. Hence, it follows that
  $\cl(A^*\Cscr')$ also does not contain the origin.

  Now apply Proposition~\ref{sec3:antipolarAC} with $A^*$ in place
  of $A$, and $\Cscr'$ in place of $\Cscr$, to obtain
  \begin{equation*}
    (A^*\Cscr')' = A^{-1}\Cscr''.
  \end{equation*}
  Taking the antipolar on both sides of the above relation, we arrive at
  \begin{equation}\label{sec3:2sidepolar}
    (A^*\Cscr')'' = (A^{-1}\Cscr'')'.
  \end{equation}
  Since $\Cscr'$ is ray-like, it follows that $\cl(A^*\Cscr')$ is
  also ray-like. Since $\cl(A^*\Cscr')$ does not contain the
  origin, we conclude from the bi-antipolar theorem that $(A^*\Cscr')'' = \cl(A^*\Cscr')$. Moreover, we have
  \[
  \big(A^{-1}\Cscr''\big)' =
  \Bigg(\bigcup_{\lambda\ge 1}\lambda
  A^{-1}\Cscr
  \Bigg)' = \big(A^{-1}\Cscr\big)',
  \]
  where the first equality follows from Proposition~\ref{sec3:C''},
  and the second equality can be verified directly from
  definition. The conclusion now follows from the above discussion and
  \eqref{sec3:2sidepolar}.
\end{proof}

We have the following further consequence.

\begin{proposition}\label{sec3:cor2}
  Suppose that $A^{-1}\Cscr\neq \emptyset$, and either $\Cscr$ is
  polyhedral or $\ri\Cscr\cap \range A\neq\emptyset$.  Then
  $(A^{-1}\Cscr)'$ is nonempty and
  \begin{equation*}\label{sec3:antipolarinvAC}
    \big(A^{-1}\Cscr\big)' = A^*\Cscr'.
  \end{equation*}
\end{proposition}

\begin{proof}
  We will show that $A^*\Cscr'$ is closed under the assumption of this
  proposition.  Then the conclusion follows immediately from
  Proposition~\ref{sec3:cor1}.

  Abrams's theorem \cite[Lemma~3.1]{Berman:1973} asserts that
  $A^*\Cscr'$ is closed if and only if $\Cscr' + \ker A^*$ is
  closed. We will thus establish the closedness of the latter set.

  Suppose that $\Cscr$ is a polyhedral. Then it is routine to show
  that $\Cscr'$ is also a polyhedral and thus $\Cscr' + \ker A^*$ is
  closed. Hence, the conclusion of the corollary holds under this
  assumption.

  Finally, suppose that $\ri \Cscr\cap \range A\neq\emptyset$.  From
  \citet[Theorem~2.2.1]{AuT03} and the bipolar theorem, we have
  $\cl\dom\sigma_{\Cscr'} = [(\Cscr')_\infty]^\circ$, where
  $(\Cscr')_\infty$ is the recession cone of $\Cscr'$, which turns out
  to be just $\Cscr^*$ by Lemma~\ref{sec4:lem1}. From this and the
  bipolar theorem, we see further that
  \[
  \cl\dom\sigma_{\Cscr'} = [\Cscr^*]^\circ = [(\cl\cone\Cscr)^*]^\circ = -\cl\cone\Cscr,
  \]
  and hence $\ri\dom\sigma_{\Cscr'} = -\ri\cone\Cscr$, thanks to
  \citet[Theorem~6.3]{Roc70}. Furthermore, the assumption that
  $\ri\Cscr\cap \range A\neq \emptyset$ is equivalent to
  $\ri\cone\Cscr\cap \range A\neq \emptyset$, since $\ri\cone\Cscr =
  \bigcup_{\lambda > 0}\lambda\ri\Cscr$; see
  \citet[p.~50]{Roc70}. Thus, the assumption $\ri\Cscr\cap \range
  A\neq\emptyset$ together with \citet[Theorem~23.8]{Roc70} imply that
  \begin{equation*}\label{sec3:sum1}
    \Cscr' + \ker A^* = \partial \sigma_{\Cscr'}(0) + \partial\delta_{\range  A}(0) = \partial(\sigma_{\Cscr'} + \delta_{\range A})(0).
  \end{equation*}
  In particular, $\Cscr' + \ker A^*$ is closed.
\end{proof}

\subsection{Unions and intersections}

Other important set operations are union and intersection, which we
discuss here.  \citet[Appendix~A.1]{RuysW79} outline additional rules.

\begin{proposition}\label{sec3:union}
  Let $\Cscr_1$ and $\Cscr_2$ be nonempty closed convex sets. Then
  \[
  (\Cscr_1\cup \Cscr_2)' = \Cscr_1'\cap \Cscr_2'.
  \]
  If $0\notin \cl\conv(\Cscr_1\cup \Cscr_2)$, then the sets above
  are nonempty.
\end{proposition}

\begin{proof}
  Note that $y\in (\Cscr_1\cup \Cscr_2)'$ is equivalent to
  $\innerp{y}{x} \ge 1$ for all $x\in \Cscr_1$ as well as $x\in
  \Cscr_2$.  This is equivalent to $y\in \Cscr_1'\cap \Cscr_2'$.
  Moreover, if we assume further that $0\notin \cl\conv(\Cscr_1\cup
  \Cscr_2)$, then $(\Cscr_1\cup \Cscr_2)' = [\cl\conv(\Cscr_1\cup
  \Cscr_2)]'$ is nonempty. This completes the proof.
\end{proof}

We now consider the antipolar of intersections. Note that it is
necessary to assume that both $\Cscr_1$ and $\Cscr_2$ are ray-like,
which was missing from \citet[Property~A.5]{RuysW79}. (The necessity
of this assumption is demonstrated by
Example~\ref{ex:intersection-ray-like}, which follows the
proposition.)

\begin{proposition}\label{sec3:intersection}
  Let $\Cscr_1$ and $\Cscr_2$ be nonempty ray-like closed convex sets
  not containing the origin. Suppose further that $\Cscr_1\cap
  \Cscr_2\neq \emptyset$. Then
  \[
  (\Cscr_1\cap \Cscr_2)' = \cl\conv(\Cscr_1'\cup \Cscr_2'),
  \]
  and both sets are nonempty.
\end{proposition}

\begin{proof}
  From the fact that both $\Cscr_1$ and $\Cscr_2$ are closed convex sets not containing the origin,
  it follows that $\Cscr_1'$ and $\Cscr_2'$ are nonempty and hence $\cl{\rm conv}(\Cscr_1'\cup \Cscr_2')\neq \emptyset$.
  Moreover, because $\Cscr_1\cap \Cscr_2$ does not contain
  the origin, $(\Cscr_1\cap \Cscr_2)'$ is also nonempty.

  We first show that $\cl{\rm conv}(\Cscr_1'\cup \Cscr_2')$ does not
  contain the origin. To this end, let $y\in \Cscr_1'\cup
  \Cscr_2'$. For any $x\in \Cscr_1\cap \Cscr_2$, we have
  $\innerp{y}{x} \ge 1$, which shows that $\Cscr_1'\cup
  \Cscr_2'\subseteq (\Cscr_1\cap \Cscr_2)'$, and hence
  $\cl\conv(\Cscr_1'\cup \Cscr_2')\subseteq (\Cscr_1\cap \Cscr_2)'$.
  Since $\Cscr_1\cap \Cscr_2$ is nonempty, $(\Cscr_1\cap \Cscr_2)'$
  does not contain the origin. Consequently, $\cl\conv(\Cscr_1'\cup
  \Cscr_2')$ does not contain the origin, as claimed.

  Now apply Proposition~\ref{sec3:union}, with $\Cscr_1'$ in place of
  $\Cscr_1$ and $\Cscr_2'$ in place of $\Cscr_2$, to obtain
  \[
  (\Cscr_1'\cup \Cscr_2')' = \Cscr_1''\cap \Cscr_2'' = \Cscr_1\cap
  \Cscr_2.
  \]
  Take the antipolar of both sides to obtain
  \[
  (\Cscr_1\cap \Cscr_2)' = (\Cscr_1'\cup \Cscr_2')'' =
  [\cl\conv(\Cscr_1'\cup \Cscr_2')]'' = \cl\conv(\Cscr_1'\cup
  \Cscr_2'),
  \]
  where the second equality follows from the definition of antipolar,
  and the third equality follows from the observation that
  $\cl\conv(\Cscr_1'\cup \Cscr_2')$ is a nonempty ray-like closed
  convex set not containing the origin. This completes the proof.
\end{proof}

The following counter-example shows that the requirement that
$\Cscr_1$ and $\Cscr_2$ are ray-like cannot be removed from
Proposition~\ref{sec3:intersection}.

\begin{example}[Set intersection and the ray-like property]
  \label{ex:intersection-ray-like}
  Consider the sets
  \[
  \Cscr_1 = \set{(x_1,x_2)|1-x_1\le x_2\le x_1-1}
  \textt{and}
  \Cscr_2 = \set{(x_1,x_2)| x_1=1}.
  \]
  Define $H_1 = \set{(x_1,x_2)| x_1+x_2\ge 1}$ and $H_2 =
  \set{(x_1,x_2)| x_1-x_2\ge 1}$ so that $\Cscr_1 = H_1\cap H_2$. Clearly
  the set $\Cscr_2$ is not ray-like, while the sets $\Cscr_1$, $H_1$,
  and $H_2$ are. Moreover, all four sets do not contain the origin.
  Furthermore, $\Cscr_1\cap \Cscr_2$ is the singleton $\set{(1,0)}$,
  and hence a direct computation shows that $(\Cscr_1\cap \Cscr_2)' =
  \set{(y_1,y_2)|y_1\ge 1}$.

  Next, it follows directly from the antipolar definition that $\Cscr_2' =
  \set{(y_1,0)|y_1\ge 1}$.  Also note that $H_1 =
  L_1^{-1}I$, where $L_1(x_1,x_2) = x_1+x_2$ and $I = \set{u|u\ge
    1}$. Thus, by Proposition~\ref{sec3:cor2}, $H_1' =
  \set{(y_1,y_1)|y_1\ge 1}$.  Similarly, $H_2' =
  \set{(y_1,-y_1)|y_1\ge 1}$. Because $H_1$ and $H_2$ are ray-like, it follows from Proposition~\ref{sec3:intersection} that
  \begin{equation*}
    \begin{split}
      \Cscr_1' = (H_1\cap H_2) ' = \cl\conv(H_1'\cup H_2'),
    \end{split}
  \end{equation*}
  which contains $\Cscr_2'$.  Thus,
  \[
  \cl\conv(\Cscr_1'\cup \Cscr_2') = \Cscr_1'\subsetneq \set{(y_1,y_2)|y_1\ge 1} = (\Cscr_1\cap \Cscr_2)'.
  \]
\end{example}
\vspace{-\baselineskip}

\section{Duality derivations}
\label{sec:duality}

We derive in this section the gauge and Lagrange duals of the primal
problem~\eqref{eq:7}. Let
\begin{equation}\label{sec4:C}
  \Cscr = \set{x|\rho(b-Ax)\le \sigma}
\end{equation}
denote the constraint set, where $\rho$ is a closed gauge and $0\le \sigma < \rho(b)$. We also consider the associated
set
\begin{equation}\label{sec4:C0}
  \Cscr_0 = \set{u|\rho(b-u)\le \sigma},
\end{equation}
and note that $\Cscr = A^{-1}\Cscr_0$.
Recall from our blanket assumption in \S\ref{sec:background} that when $\sigma = 0$, we only consider
closed gauges $\rho$ with $\rho^{-1}(0) = \set0$.

\subsection{The gauge dual} \label{sec:gauge-duality}

We consider two approaches for deriving the gauge dual
of~\eqref{eq:7}. The first uses explicitly the abstract definition of
the gauge dual~\eqref{g-dual}. The second approach redefines the
objective function to also contain an indicator for the nonlinear
gauge $\rho$ where $\Cscr$ is an affine set. This alternative approach
is instructive, because it illustrates the modeling choices that are
available when working with gauge functions.

\subsubsection{First approach}

The following combines Proposition~\ref{sec3:cor2} with
Proposition~\ref{sec3:prop2}, and gives an explicit expression for the
antipolar of $\Cscr$ when $\sigma>0$.

\begin{corollary} \label{cor:rho-antipolar} Suppose that $\Cscr$ is
  given by~\eqref{sec4:C}, where $0<\sigma<\rho(b)$, and $\Cscr_0$
  is given by~\eqref{sec4:C0}.  If $\Cscr_0$ is polyhedral, or
  $\ri\Cscr_0\cap\range A\neq \emptyset$, then
  \[
  \Cscr' = \set{A^*y |  \innerp{b}{y} - \sigma\rho^\circ(y)\ge 1}.
  \]
\end{corollary}

As an aside, we present the following
result, which follows from
Corollary~\ref{sec3:prop2immedcor} and Proposition~\ref{sec3:cor2}. It concerns a general closed convex cone $\Kscr$.

\begin{corollary}
  \label{cor:rho-axb-antipolar}
  Suppose that $\Cscr = \set{x |  Ax - b\in \Kscr}$ for some closed convex
  cone $\Kscr$ and $b\notin -\Kscr$. If $\Kscr$ is polyhedral, or $(b + \ri\Kscr)\cap\range A\neq \emptyset$, then
  \[
  \Cscr' =\set{A^*y |  \innerp{b}{y} \ge 1,\ y\in \Kscr^*}.
  \]
\end{corollary}

These results can be used to obtain an explicit representation of the
gauge dual problem. We rely on the antipolar calculus developed in
\S\ref{sec:antipolar}. Assume that
\begin{equation}\label{eq:dual-assump}
  \hbox{$\Cscr_0$ is polyhedral,}
    \textt{or}
  \ri\Cscr_0\cap\range A\neq \emptyset.
\end{equation}
Consider separately the cases $\sigma > 0$ and $\sigma
= 0$.

\paragraph{Case 1: $\sigma > 0$} Apply
Corollary~\ref{cor:rho-antipolar} to derive the antipolar set
\begin{equation}\label{sec4:rep1}
  \Cscr' = \set{A^*y | \langle b,y\rangle - \sigma\rho^\circ(y)\ge 1}.
\end{equation}

\paragraph{Case 2: $\sigma = 0$} Here we use the blanket assumption (see \S\ref{sec:background}) that $\rho^{-1}(0) = \set{0}$, and in that case, $\Cscr = \set{x|Ax= b}$. Apply Corollary~\ref{cor:rho-axb-antipolar} with $\Kscr=\set{0}$ to obtain
\begin{equation}\label{sec4:rep2}
  \Cscr' = \set{A^*y | \langle b,y\rangle\ge 1}.
\end{equation}
Since $\rho^{-1}(0)=\set{0}$ and $\rho$ is
closed, we conclude from
Proposition~\ref{sec2:prop1}(v) that $\dom\rho^\circ
= \Xscr$. Hence, \eqref{sec4:rep2} can be seen as a special case of
\eqref{sec4:rep1} with $\sigma = 0$.

These two cases can be combined, and we see that when \eqref{eq:dual-assump} holds, the gauge dual problem
\eqref{g-dual} for \eqref{eq:7} can be expressed as \eqref{g-dual2}.
If the assumptions~\eqref{eq:dual-assump} are not satisfied,
then in view of Proposition~\ref{sec3:cor1}, it still holds that
\eqref{g-dual} is equivalent to
\begin{equation*}
  \minimize{u,y}\quad \polar\kappa(u)
  \quad\st\quad
  u \in \cl\set{A^*y|\innerp y b - \sigma\polar\rho(y) \ge 1}.
\end{equation*}
This optimal value can in general be less than or equal to that of
\eqref{g-dual2}.

\subsubsection{Second approach}

This approach does not rely on assumptions~\eqref{eq:dual-assump}.
Define the function $\xi(x,r,\tau):=\kappa(x)+\delta_{\epi\rho}(r,\tau)$, which is a gauge because $\epi\rho$ is a cone. Then~\eqref{eq:7} can be equivalently reformulated as
\begin{equation}\label{gprimal:alternativet}
  \minimize{x,r,\tau}\quad \xi(x,r,\tau)
  \quad\st\quad
  Ax+r=b,\ \tau=\sigma.
\end{equation}
Invoke Proposition~\ref{sec2:prop2.5} to obtain
\begin{equation*}
\begin{aligned}
  \xi^\circ(z,y,\alpha)
    &=\max\set{\kappa^\circ(z),\ (\delta_{\epi\rho})^\circ(y,\alpha)}
  \\&\overset{(i)}{=}\max\set{\kappa^\circ(z),\ \delta_{(\epi\rho)^\circ}(y,\alpha)}\\
    &\overset{(ii)}{=}\kappa^\circ(z)+\delta_{(\epi\rho)^\circ}(y,\alpha)
    \overset{(iii)}{=}\kappa^\circ(z)+\delta_{\epi(\rho^\circ)}(y,-\alpha),
\end{aligned}
\end{equation*}
where (i) follows from Proposition~\ref{sec2:prop2}(i), (ii) follows from
the definition of indicator function, and (iii) follows from Proposition~\ref{sec2:prop1}(vi).
As \citet[\S2]{freund:1987} shows for gauge programs with linear constraints, the gauge dual is given by
\begin{equation*}\label{gdual:alternativet}
  \minimize{y,\alpha}\quad \xi^\circ(A^*y,y,\alpha)
  \quad\st\quad
  \langle y,b\rangle+\sigma\alpha\ge1,
\end{equation*}
which can be rewritten as
\begin{equation*}\label{gdual:alternativett}
  \minimize{y,\alpha}\quad \kappa^\circ(A^*y)
  \quad\st\quad
  \langle y,b\rangle+\sigma\alpha\ge1,\ \rho^\circ(y)\leq-\alpha.
\end{equation*}
(The gauge dual for problems with linear constraints also follows
directly from Corollary~\ref{cor:rho-axb-antipolar} with
$\Kscr=\set0$.)  Further simplification leads to the gauge dual
program~\eqref{g-dual2}.

Note that the transformation used to derive~\eqref{gprimal:alternativet} is
very flexible. For example, if~\eqref{eq:7} contained the additional
conic constraint $x\in\Kscr$, then $\xi$ could be defined to contain an
additional term given by the indicator of $\Kscr$.

Even though this approach does not require the assumptions~\eqref{eq:dual-assump} used in \S\ref{sec:gauge-duality}, and thus appears to apply more generally, it is important to keep in mind that we have yet to impose conditions that imply strong duality. In fact, as we show in \S\ref{sec:strong-duality}, the assumptions required there imply
\eqref{eq:dual-assump}.

\subsection{Lagrange duality} \label{sec:lagrange-duality}

Our derivation of the Lagrange dual problem~\eqref{l-dual} is standard, and we include here as a counterpoint to the corresponding gauge dual derivation. We begin by
reformulating
\eqref{eq:7} by introducing an artificial variable $r$, and deriving the dual
of the equivalent problem
\begin{equation}\label{rprimal}
  \minimize{x,r}\quad \kappa(x)
  \quad\st\quad
  Ax+r=b,\ \rho(r)\leq\sigma.
\end{equation}
Define the Lagrangian function
\begin{align*}
  L(x,r, y)
  &=\kappa(x)+\innerp{y}{b-Ax-r}.
\end{align*}

The Lagrange dual problem is given by
\begin{equation*}
  \maximize{y}\quad
  \inf_{x,\,\rho(r)\le \sigma}\ L(x,r, y).
\end{equation*}
Consider the (concave) dual function
\begin{align*}
  \ell( y)
  &=\inf_{x,\,\rho(r)\le \sigma} L(x,r, y)\\
  &=\inf_{x,\,\rho(r)\le \sigma}\Bigl\{\innerp{y}{b} - \innerp{y}{r}
  -\bigl(\innerp{A^* y}{x}-\kappa(x)\bigr)\Bigr\}\\
  &=\innerp{y}{b} - \sup_{\rho(r)\le \sigma} \innerp{y}{r}
  -\sup_{x}\Bigl\{\langle A^* y,x\rangle-\kappa(x)\Bigr\}\\
  &=\innerp{y}{b}-\sigma\rho^\circ(y)
  -\delta_{\kappa^\circ(\cdot)\leq1}(A^* y),
\end{align*}
where the first conjugate on the right-hand side follows from Proposition~\ref{sec2:prop1}(iii) when $\sigma > 0$,
and when $\sigma = 0$, it is a direct consequence of the assumption that $\rho^{-1}(0)=\set0$ so that $\dom\rho^\circ=\Xscr$ from
Proposition~\ref{sec2:prop1}(v);
the last conjugate follows from Proposition~\ref{sec2:prop1}(iv). The Lagrange dual problem is obtained by maximizing $\ell$,
leading to \eqref{l-dual}.

Strictly speaking, the Lagrangian primal-dual pair of problems that we
have derived is given by~\eqref{rprimal} and~\eqref{l-dual}, but it is
easy to see that \eqref{eq:7} is equivalent to \eqref{rprimal}.
in the sense that the respective optimal values are
the same, and that solutions to one problem readily lead to solutions
for the other.
Thus, without loss of generality, we refer to~\eqref{l-dual} as the Lagrange dual to the primal
problem~\eqref{eq:7}.

\section{Strong duality} \label{sec:strong-duality}

Freund's \citeyear{freund:1987} analysis of the gauge dual pair is mainly based on the classical separation theorem. It relies %
on the ray-like property of the constraint set $\Cscr$. Our study of the gauge dual pairs allows us to relax the ray-like assumption. By establishing connections with the Fenchel duality framework, we can develop strong duality conditions that are analogous to those required for Lagrange duality theory.

The Fenchel dual \cite[\S31]{Roc70} of \eqref{g-primal} is given by
\begin{equation}
  \label{sec4:eq1}
  \maximize{y} \quad -\sigma_{\Cscr}(-y) \textt{subject to} \kappa^\circ(y)\le 1,
\end{equation}
where we use $(\delta_\Cscr)^*=\sigma_\Cscr$ and Proposition~\ref{sec2:prop1}(iv) to obtain $\kappa^*=\delta_{[\kappa^\circ \le 1]}$.
Let $v_p$, $v_g$, and $v_f$, respectively, denote the optimal values of \eqref{g-primal}, \eqref{g-dual} and \eqref{sec4:eq1}.
The following result relates their optimal values and dual solutions.

\begin{theorem}[Weak duality]\label{thm:weak-duality}
  Suppose that $\dom\kappa^\circ\cap \Cscr'\neq \emptyset$.
  Then
  \[
  v_p\ge v_{f} = 1/v_g > 0. %
  \]
  Furthermore,
  \begin{enumerate}[{\rm (i)}]
  \item if $y^*$ solves \eqref{sec4:eq1}, then $y^*\in \cone\Cscr'$
  and $y^*/v_f$ solves \eqref{g-dual};
  \item if $y^*$ solves \eqref{g-dual} and $v_g > 0$, then $v_f y^*$ solves \eqref{sec4:eq1}.
  \end{enumerate}
\end{theorem}

\begin{proof}
  The fact that $v_p\ge v_f$ follows from standard Fenchel duality
  theory. We now show that $v_{f} = 1/v_{g}$.

  Because $\dom\kappa^\circ\cap \Cscr'\neq \emptyset$, there exists $y_0$ such that $\kappa^\circ(y_0)\le 1$
  and $y_0\in \tau\Cscr'$ for some $\tau>0$. In particular, becuse $-\sigma_\Cscr(-y)= \inf_{c\in \Cscr}\innerp{c}{y}$ for all $y$,
  it follows from the definition of $v_f$ that
  \begin{equation}\label{sec4:eq3}
    v_{f} = \sup_y\set{\inf_{c\in \Cscr}\innerp{c}{y}|\kappa^\circ(y)\le 1}\ge \inf_{c\in \Cscr} \innerp{c}{y_0} \ge \tau > 0.
  \end{equation}
  Hence
  \begin{equation}\label{sec4:forcite}
      v_{f} = \sup_{y,\lambda}\set{\lambda | \kappa^\circ(y)\le 1,\ -\sigma_{\Cscr}(-y) \ge \lambda,\
         \lambda > 0}.
  \end{equation}
  From this,
  we have further that
  \[
    \begin{aligned}
      v_{f}
      &= \sup_{y,\lambda}\set{\lambda |  \kappa^\circ(y/\lambda)\le 1/\lambda,\ -\sigma_{\Cscr}(-y/\lambda) \ge 1, 1/\lambda > 0}
    \\&= \sup_{y,\mu}\set{1/\mu | \kappa^\circ(\mu y)\le\mu,\ -\sigma_{\Cscr}(-\mu y) \ge
      1,\ \mu > 0}.
    \end{aligned}
  \]
  Inverting both sides of this equation gives
  \begin{equation}\label{sec4:rel1}
    \begin{aligned}
      1/v_{f}
      &=\inf_{y,\mu}\set{\mu|\kappa^\circ(\mu y)\le\mu,\ -\sigma_{\Cscr}(-\mu y) \ge 1,\ \mu > 0}
    \\&=\inf_{w,\mu}\set{\mu|\kappa^\circ(w)\le\mu,\ -\sigma_{\Cscr}(-w) \ge 1,\ \mu > 0}
    \\&\overset{\rlap{\scriptsize\it(i)}}{=}\inf_{w,\mu}\set{\mu|\kappa^\circ(w)\le\mu,\ w\in \Cscr',\ \mu > 0}
    \\&=\inf_{w,\mu}\set{\mu|\kappa^\circ(w)\le\mu,\ w\in \Cscr'}
    \\&=\inf_{w}\set{\kappa^\circ(w) | w\in \Cscr'} = v_{g},
    \end{aligned}
  \end{equation}
  where equality (i) follows from the definition of $\Cscr'$.  This proves $v_{f} = 1/v_{g}$.

  We now prove item (i). Assume that $y^*$ solves \eqref{sec4:eq1}. Then $v_{f}$ is
  nonzero (by \eqref{sec4:eq3}) and finite, and so is $v_{g} = 1/v_{f}$.  Then $y^*\in\cone \Cscr'$
  because $-\sigma_{\Cscr}(-y^*) = \inf_{c\in \Cscr}\innerp{c}{y^*} = v_{f} > 0$, and we see
  from \eqref{sec4:rel1} that $y^*/v_f$ solves
  \eqref{g-dual}. We now prove item (ii). Note that if $y^*$ solves \eqref{g-dual}
  and $v_{g} > 0$, then $\kappa^\circ(y^*) > 0$. One can then observe
  similarly from \eqref{sec4:rel1} that $y^*/v_g = v_f y^*$ solves
  \eqref{sec4:eq1}. This completes the proof.
\end{proof}

Fenchel duality theory allows us to use Theorem~\ref{thm:weak-duality} to obtain several sufficient
conditions that guarantee strong duality, i.e., $v_pv_{g}=1$, and the attainment of the gauge
dual problem \eqref{g-dual}.
For example, applying \citet[Theorem~31.1]{Roc70} yields the following
corollary.

\begin{corollary}[Strong duality I]\label{sec4:strongduality}
  Suppose that $\dom\kappa^\circ\cap \Cscr'\neq \emptyset$ and $\ri\dom\kappa\cap \ri\Cscr\neq \emptyset$.
  Then $v_pv_{g}=1$ and the gauge dual \eqref{g-dual} attains its optimal value.
\end{corollary}

\begin{proof}
  From $\ri\dom\kappa\cap \ri\Cscr\neq \emptyset$
  and \citet[Theorem~31.1]{Roc70}, we see that $v_p = v_{f}$ and
  $v_{f}$ is attained. The conclusion of the corollary now follows
  immediately from Theorem~\ref{thm:weak-duality}.
\end{proof}

We would also like to guarantee {\em primal attainment}. Note that the
gauge dual of the gauge dual problem \eqref{g-dual} (i.e., the bidual
of \eqref{g-primal}) is given by
\begin{equation}
  \label{sec4:eq4}
  \minimize{x} \quad \kappa^{\circ\circ}(x) \textt{subject to} x\in\Cscr'',
\end{equation}
which is not the same as \eqref{g-primal} unless $\Cscr$ is ray-like
and $\kappa$ is closed; see Theorem~\ref{thm:bi-antipolar} and Proposition~\ref{sec2:prop1}(ii).  However, we show in the next proposition that
\eqref{sec4:eq4} and \eqref{g-primal} always have the same optimal
value when $\kappa$ is closed (even if $\Cscr$ is not ray-like), and that if the optimal value is
attained in one problem, it is also attained in the other.

\begin{proposition}\label{sec4:prop1}
  Suppose that $\kappa$ is closed. Then the optimal values of
  \eqref{g-primal} and \eqref{sec4:eq4} are the same. Moreover, if the
  optimal value is attained in one problem, it is also attained in the
  other.
\end{proposition}

\begin{proof}
  From Proposition~\ref{sec3:C''}, we see that \eqref{sec4:eq4} is equivalent to
  \begin{equation*}\label{sec4:primal2}
    \minimize{\lambda,x} \quad \lambda\kappa(x) \textt{subject to} x\in\Cscr,\ \lambda\ge 1,
  \end{equation*}
  which clearly gives the same optimal value as \eqref{g-primal}. This
  proves the first conclusion.  The second conclusion now also follows
  immediately.
\end{proof}

Hence, we obtain the following corollary, which generalizes
\citet[Theorem~2A]{freund:1987} by dropping the ray-like assumption
on $\Cscr$.

\begin{corollary}[Strong duality II]\label{sec4:corpdzerogap}
  Suppose that $\kappa$ is closed, and that $\ri\dom\kappa\cap
  \ri\Cscr\neq \emptyset$ and $\ri\dom\kappa^\circ\cap \ri\Cscr'\neq
  \emptyset$.  Then $v_pv_{g}=1$ and both values are attained.
  \end{corollary}
  \begin{proof}
    The conclusion follows from Corollary~\ref{sec4:strongduality}, Proposition~\ref{sec4:prop1}, the fact
    that $\kappa = \kappa^{\circ\circ}$ for closed gauge functions,
    and the observation that $\ri\dom\kappa\cap \ri\Cscr\neq \emptyset$
    if and only if $\ri\dom\kappa\cap \ri\Cscr''\neq \emptyset$, since
    $\ri\Cscr'' = \bigcup_{\lambda > 1}\lambda \ri\Cscr$ \citep[p.~50]{Roc70}
    and $\dom\kappa$ is a cone.
  \end{proof}

Before closing this section, we specialize Theorem~\ref{thm:weak-duality} to study the relationship between the Lagrange~\eqref{l-dual} and gauge~\eqref{g-dual2} duals. Let $v_l$ denote the optimal value of \eqref{l-dual}. We use the fact that, for any $y$,
  \begin{equation}\label{sec4:eq}
    -\sigma_\Cscr(-y) = \inf_{c\in \Cscr}\ \innerp{c}{y} \begin{cases}
      >0& {\rm if\ }y\in \cone\Cscr'\backslash\set{0},\\
      \le 0& {\rm otherwise},
    \end{cases}
  \end{equation}
which is directly verifiable using the definition of $\Cscr'$.

\begin{corollary}\label{sec4:cor4.7}
  Suppose that $\Cscr$ is given by~\eqref{sec4:C}, where $0\le\sigma<\rho(b)$, assumption~\eqref{eq:dual-assump} holds,
  and $\dom \kappa^\circ\cap \Cscr'\neq \emptyset$. Then $v_l = v_f > 0$. Moreover,
  \begin{enumerate}[{\rm (i)}]
    \item if $y^*$ solves \eqref{l-dual}, then $y^*/v_l$ solves \eqref{g-dual2};
    \item if $y^*$ solves \eqref{g-dual2}
  and $v_{g}> 0$, then $v_{l}y^*$ solves \eqref{l-dual}.
  \end{enumerate}
\end{corollary}
\begin{proof}
  From \eqref{sec4:eq}, for any $y\in \cone\Cscr'\backslash\set{0}$,
  we have $-\sigma_\Cscr(-y) = \inf_{c\in \Cscr}\innerp{c}{y} > 0$ and is hence finite.
  Note that $\inf_{c\in \Cscr}\innerp{c}{y} = \inf_{c,r}\set{\innerp{c}{y}| Ac +r = b,\ \rho(r)\le \sigma}$.
  Use this reformulation and proceed as in \S\ref{sec:lagrange-duality} to obtain the dual function
  \begin{align*}
  \ell(u)
  &=\inf_{c,\,\rho(r)\le \sigma}\left\{\innerp{u}{b} - \innerp{u}{r}
  -\bigl(\innerp{A^* u}{c}-\innerp{c}{y}\bigr)\right\}\\
  &=\innerp{b}{u} - \sup_{\rho(r)\le \sigma}\, \innerp{u}{r}
  -\sup_{c}\left\{\langle A^* u - y,c\rangle\right\}\\
  & = \innerp{b}{u} - \sigma\rho^\circ(u) - \delta_{A^*u=y}(u).
  \end{align*}
  The dual problem to $\inf_{c\in \Cscr}\innerp{c}{y}$ is given by maximizing $\ell$ over $u$. Because of assumption~\eqref{eq:dual-assump}
  and the finiteness of $-\sigma_\Cscr(-y)$,
  \begin{equation}\label{gref}
  \inf_{c\in {\cal C}}\ \innerp c y = \sup_{y = A^*u}
  \, \big\{\langle b,u\rangle - \sigma\rho^\circ(u)\big\},
  \end{equation}
  and the supremum is attained, which is a consequence of
  \citet[Corollary~28.2.2 and Theorem~28.4]{Roc70}. On the other hand, for any
  $y\notin \cone\Cscr'\backslash\set{0}$, we have from weak duality
  and~\eqref{sec4:eq} that
  \begin{equation}\label{gle0}
    \sup_{y = A^*u}\left\{\langle b,u\rangle - \sigma\rho^\circ(u)\right\}
    \le \inf_{c\in {\cal C}}\ \langle c,y\rangle \le 0.
  \end{equation}

  Since $\dom \kappa^\circ\cap \Cscr'\neq \emptyset$, we can substitute
  \eqref{gref} into \eqref{sec4:forcite} and obtain
  \begin{equation*}
  \begin{split}
    0<v_{f} &= \sup\set{\lambda |  \kappa^\circ(y)\le 1,\ -\sigma_\Cscr(-y)\ge \lambda > 0} \\
    &=\sup\set{\langle b,u\rangle - \sigma\rho^\circ(u)|  \kappa^\circ(A^*u)\le 1,\ A^*u\in \cone\Cscr'\backslash\set0}\\
    &=\sup\set{\langle b,u\rangle - \sigma\rho^\circ(u)|  \kappa^\circ(A^*u)\le 1}= v_{l},
  \end{split}
  \end{equation*}
  where the last equality follows from \eqref{gref}, \eqref{gle0}, and the
  positivity of $v_f$. This completes the first part of the proof. In
  particular, the Fenchel dual problem \eqref{sec4:eq1} has the same optimal
  value as the Lagrange dual problem \eqref{l-dual}, and $y^* = A^*u^*$ solves
  \eqref{sec4:eq1} if and only if $u^*$ solves
  \eqref{l-dual}. Moreover, since assumption~\eqref{eq:dual-assump} holds, \S\ref{sec:gauge-duality} shows that \eqref{g-dual}
  is equivalent to \eqref{g-dual2}. The conclusion now follows from these and
  Theorem~\ref{thm:weak-duality}.
\end{proof}

We next state a strong duality result concerning the primal-dual gauge pair
\eqref{eq:7} and \eqref{g-dual2}.

\begin{corollary}\label{sec4:cor4.8}
  Suppose that $\Cscr$ and $\Cscr_0$ are given by~\eqref{sec4:C} and~\eqref{sec4:C0}, where $0\le\sigma<\rho(b)$. Suppose also that $\kappa$ is closed,
    \begin{equation}\label{cor5.8CQ}
    \ri\dom\kappa
  \,\cap A^{-1}\ri\Cscr_0\neq \emptyset,\textt{and}\ri\dom\kappa^\circ\cap
  A^*\ri\Cscr_0'\neq \emptyset.
  \end{equation}
  Then the optimal values of \eqref{eq:7} and \eqref{g-dual2} are attained, and their product is $1$.
\end{corollary}

\begin{proof}
  Since $A^{-1}\ri\Cscr_0\neq \emptyset$, $A$ satisfies the
  assumption in \eqref{eq:dual-assump}. Then \S\ref{sec:gauge-duality} shows that \eqref{g-dual}
  is equivalent to \eqref{g-dual2}.  Moreover, from \citet[Theorem~6.6,
  Theorem~6.7]{Roc70}, we see that $\ri\Cscr = A^{-1}\ri\Cscr_0$ and
  $\ri\Cscr' = A^*\ri\Cscr_0'$. The conclusion now follows from
  Corollary~\ref{sec4:corpdzerogap}.
\end{proof}

This last result also holds if $\Cscr_0$ were polyhedral; in that case, the assumptions \eqref{cor5.8CQ} could be replaced with $\ri\dom\kappa
  \,\cap \Cscr\neq \emptyset$ and $\ri\dom\kappa^\circ\cap
 \Cscr'\neq \emptyset$.

\section{Variational properties of the gauge value function}
\label{sec:sensitivity-analysis}

Thus far, our analysis has focused on the relationship between the
optimal values of the primal-dual pair~\eqref{eq:7}
and~\eqref{g-dual2}. As with Lagrange duality, however, there is also
a fruitful view of dual solutions as providing sensitivity information
on the primal optimal value. Here we provide a corresponding variational
analysis of the gauge optimal-value function with respect to
perturbations in $b$ and $\sigma$.

Sensitivity information is captured in the subdifferential of the
value function
\begin{equation}\label{sec4:value}
  v(h,k) = \inf_x\ f(x,h,k),
\end{equation}
with
\begin{equation}\label{sec4:f}
  f(x,h,k) = \kappa(x) + \delta_{\epi\rho}(b+h-Ax,\sigma+k).
\end{equation}
Following the discussion
in \citet[Section~4]{AravkinBurkeFriedlander:2013}, we start by
computing the conjugate of $f$, which can be done as follows:
\[
  \begin{aligned}
    f^*(z,y,\tau)
    &= \sup_{x,h,k}\set{\langle z,x\rangle+\langle y,h\rangle+\tau k - \kappa(x) - \delta_{\epi\rho}(b+h-Ax,\sigma+k)}
  \\&= \sup_{x,w,\mu}\set{\langle z+A^*y,x\rangle - \kappa(x)+\langle y,w\rangle+\tau\mu - \delta_{\epi\rho}(w,\mu)}-\innerp{b}{y} - \tau \sigma
  \\&= \kappa^*(z+A^*y) + \delta_{\epi\rho}^*(y,\tau)-\innerp{b}{y} - \tau \sigma.
  \end{aligned}
\]
Use Proposition~\ref{sec2:prop1}(iv) and the definition of support
function and convex conjugate to further transform this as
\[
\begin{aligned}
  f^*(z,y,\tau) + \innerp{b}{y} + \tau \sigma
    &= \delta_{\kappa^\circ(\cdot)\le 1}(z+A^*y) + \sigma_{\epi\rho}(y,\tau)\\
    &\overset{(i)}{=} \delta_{\kappa^\circ(\cdot)\le 1}(z+A^*y) + \delta_{(\epi\rho)^\circ}(y,\tau)\\
    &\overset{(ii)}{=} \delta_{\kappa^\circ(\cdot)\le 1}(z+A^*y) + \delta_{\epi(\rho^\circ)}(y,-\tau)\\
    &= \delta_{\kappa^\circ(\cdot)\le 1}(z+A^*y) + \delta_{\rho^\circ(\cdot)\le \cdot}(y,-\tau),
  \end{aligned}
\]
where equality (i) follows from Proposition~\ref{sec2:prop3} and
Proposition~\ref{sec2:prop2}(i), and equality (ii) follows from Proposition~\ref{sec2:prop1}(vi).
Combining this with the definition of the value function
$v(h,k)$,
\begin{equation}\label{sec4:conjv_pre}
  \begin{aligned}
    v^*(y,\tau) &= \sup_{h,k}\set{\innerp{y}{h} + \tau k - v(h,k)}
  \\&= \sup_{x,h,k}\set{\innerp{y}{h} + \tau k - f(x,h,k)}
  \\& = f^*(0,y,\tau)
      = - \innerp{b}{y} - \sigma\tau + \delta_{\kappa^\circ(\cdot)\le 1}(A^*y) + \delta_{\rho^\circ(\cdot)\le \cdot}(y,-\tau).
  \end{aligned}
\end{equation}
In view of \citet[Theorem~11.39]{RoW98}, under a suitable constraint
qualification, the set of subgradients of $v$ is nonempty and is given
by
\begin{equation}\label{sec4:conjv}
\begin{split}
\partial v(0,0) & = \mathop{\rm argmax}_{y,\tau}\set{ - f^*(0,y,\tau)}\\
& = \mathop{\rm argmax}_{y,\tau}\set{\langle b,y\rangle + \sigma \tau| \kappa^\circ(A^*y)\le 1,\ \rho^\circ(y)\le -\tau}\\
& = \Set{\big(y,-\rho^\circ(y)\big)| y \in \mathop{\rm argmax}_{y}\set{\langle b,y\rangle - \sigma \rho^\circ(y)| \kappa^\circ(A^*y)\le 1}},
\end{split}
\end{equation}
in terms of the solution set of \eqref{l-dual} and the corresponding function value of $\rho^\circ(y)$.
We state formally this result, which is a consequence of the above discussion and Corollary~\ref{sec4:cor4.7}.

\begin{proposition}
  For fixed $(b,\sigma)$, define $v$ as in \eqref{sec4:value} and $f$ as in \eqref{sec4:f}. Then
  \[
  \dom f(\cdot,0,0) \neq \emptyset
  \quad \Longleftrightarrow\quad 0
  \in  A\dom\kappa - [\rho(b - \cdot)\le \sigma],
  \]
  and hence
  \[
  (0,0)\in \interior\dom v
  \quad \Longleftrightarrow\quad
  0 \in \interior(A\dom\kappa - [\rho(b - \cdot)< \sigma])
  \]
  If $(0,0)\in \interior\dom v$ and $v(0,0) > 0$, then $\partial v(0,0)\neq \emptyset$ with
  \begin{equation*}
  \begin{split}
  \partial v(0,0)
  &= \Set{(y,-\rho^\circ(y))| y \in \mathop{\rm argmax}_{y}\set{\langle b,y\rangle - \sigma \rho^\circ(y)| \kappa^\circ(A^*y)\le 1}}\\
  &= \Set{v(0,0)\cdot(y,-\rho^\circ(y))| y \in \mathop{\rm argmin}_{y}\set{\kappa^\circ(A^*y)| \langle b,y\rangle - \sigma \rho^\circ(y)\ge 1}}.
  \end{split}
  \end{equation*}
\end{proposition}
\begin{proof}
  It is routine to verify the properties of the domain of $f(\cdot,0,0)$ and the interior of the domain of $v$.
  Suppose that $(0,0)\in \interior\dom v$. Then the value function is continuous at $(0,0)$
  and hence $\partial v(0,0)\neq \emptyset$. The first expression of $\partial v(0,0)$
  follows directly from \citet[Theorem~11.39]{RoW98} and the discussions preceding this proposition.

  We next derive the second expression of $\partial v(0,0)$.
  Since $(0,0)\in \interior\dom v$ implies $0\in \interior(A\dom\kappa - [\rho(b - \cdot)< \sigma])$,
  the linear map $A$ satisfies assumption~\ref{eq:dual-assump}. Moreover, as
  another consequence of \citet[Theorem~11.39]{RoW98},
  $(0,0)\in \interior\dom v$ also implies that $v(0,0) = \sup_{y,\tau}\{-f^*(0,y,\tau)\}$, which is just
  the optimal value of the Lagrange dual problem \eqref{l-dual}. Furthermore, $v(0,0)$ being finite and nonzero together
  with the definition of \eqref{l-dual} and \eqref{sec4:rep1} implies that $\dom\kappa^\circ\cap \Cscr'\neq \emptyset$.
  The second expression of $\partial v(0,0)$ now follows from these three observations and Corollary~\ref{sec4:cor4.7}.
\end{proof}

\section{Extensions} \label{sec:extensions}

The following examples illustrate how to extend the canonical
formulation~\eqref{eq:7} to accommodate related problems. It also
provides an illustration of the techniques that can be used to pose
problems in gauge form and how to derive their corresponding gauge
duals.

\subsection{Composition and conic side constraints}
\label{sec:linear-map}

A useful generalization of~\eqref{eq:7} is to allow the gauge
objective to be composed with a linear map, and for the addition of
conic side constraints. The composite objective can be used to
capture, for example, problems such as weighted basis pursuit (e.g.,
\cite{CandesWakinBoyd:2008,FMSY:2012}), or together with the conic
constraint, problems such as nonnegative total variation
\citep{KLY:2007}.

The following result generalizes the canonical primal-dual gauge pair
\eqref{eq:7} and~\eqref{g-dual2}.

\begin{proposition}
 Let $D$ be a linear map and $\Kscr$ be a convex cone. The following
 pair of problems constitute a primal-dual gauge pair:
\begin{subequations}
  \begin{alignat}{4}\label{eq:12}
  &\minimize{x}\quad & &\kappa(Dx) &\quad &\st
  &\quad
  &\rho(b-Ax)\leq\sigma,\ x\in\Kscr,
\\
  \label{eq:12-dual}
  &\minimize{y,\,z} & &\kappa^\circ(z)
  & &\st
  & &\langle y,b\rangle-\sigma\rho^\circ(y)\geq1,\ D^*z-A^*y\in\Kscr^*.
\end{alignat}
\end{subequations}
\end{proposition}

\begin{proof}
  Reformulate~\eqref{eq:12} as a gauge optimization problem by
  introducing additional variables, and lifting both the cone $\Kscr$
  and the epigraph $\epi\rho$ into the objective by means of their
  indicator functions: use the function
  $f(x,s,r,\tau):=\delta_\Kscr(x)+\kappa(s)+\delta_{\epi\rho}(r,\tau)$
  to define the equivalent gauge optimization problem
\begin{equation*}
  \minimize{x,s,r,\tau} \quad f(x,s,r,\tau)
  \quad\st\quad
  Dx=s,\ Ax+r=b,\ \tau=\sigma.
\end{equation*}
As with \S\ref{sec:gauge-duality}, observe that $f$ is a sum of gauges
on disjoint variables. Thus, we invoke Proposition~\ref{sec2:prop2.5}
to deduce the polar of the above objective:
\begin{equation*}
\begin{aligned}
    f^\circ(u,z,y,\alpha)
    &=\max\set{\delta^\circ_{\Kscr}(u),\ \kappa^\circ(z),\ \delta^\circ_{\epi\rho}(y,\alpha)}\\
    &\overset{(i)}{=}\max\set{\delta_{\Kscr^\circ}(u),\ \kappa^\circ(z),\ \delta_{(\epi\rho)^\circ}(y,\alpha)}\\
    &\overset{(ii)}{=}\max\set{\delta_{\Kscr^*}(-u),\ \kappa^\circ(z),\ \delta_{\epi(\rho^\circ)}(y,-\alpha)}\\
    &\overset{(iii)}{=}\delta_{\Kscr^*}(-u)+\kappa^\circ(z)+\delta_{\epi(\rho^\circ)}(y,-\alpha),
  \end{aligned}
\end{equation*}
where (i) follows from Proposition~\ref{sec2:prop2}(i), (ii) follows
from Proposition~\ref{sec2:prop1}(vi), and (iii) follows from the
definition of indicator function. Moreover, use
Corollary~\ref{cor:rho-axb-antipolar} to derive the antipolar of the
linear constraint set $\Cscr=\set{(x,s,r,\tau) | Dx=s,\, Ax+r=b,\,
  \tau=\sigma}$:
\[
\Cscr' = \set{(-D^*z + A^*y,z,y,\alpha)|\langle b,y\rangle+\sigma\alpha\geq1}.
\]
From the above discussion, we obtain the following gauge
program
\begin{equation*}
  \minimize{y,z,\alpha}
  \quad
  \delta_{\Kscr^*}(D^*z-A^*y)
  +\kappa^\circ(z)+\delta_{\epi(\rho^\circ)}(y,-\alpha)
  \quad\st\quad
  \langle b,y\rangle+\sigma\alpha\geq1.
\end{equation*}
Bringing the indicator functions down to the constraints leads to
\begin{equation*}
  \minimize{y,z,\alpha} \quad \kappa^\circ(z)
  \quad\st\quad
  \langle y,b\rangle+\sigma\alpha\geq1,\
  \rho^\circ(y)\leq-\alpha,\
  D^*z-A^*y\in\Kscr^*;
\end{equation*}
further simplification by eliminating $\alpha$ yields
the gauge dual problem~\eqref{eq:12-dual}.
\end{proof}

\subsection{Nonnegative conic optimization}

Conic optimization subsumes a large class of convex optimization
problems that ranges from linear, to second-order, to semidefinite
programming, among others. Example~\ref{example:conic-gauge} describes
how a general conic optimization problem can be reformulated as an
equivalent gauge problem; see~\eqref{eq:8}.

We can easily accommodate a generalization of~\eqref{eq:8} by embedding
it within the formulation defined by~\eqref{eq:2}, and define
\begin{equation}
  \label{eq:6}
  \minimize{x} \quad \innerp{c}{x} + \indicator\Kscr(x) \quad\st\quad \rho(b-Ax)\le\sigma,
\end{equation}
with $c\in\Kscr^{*}$, as the conic gauge optimization problem.  The
following result describes its gauge dual.

\begin{proposition} \label{thm:gauge-plus-cone}
  Suppose that
   $\Kscr\subset\Xscr$ is a convex cone and $c\in\Kscr^*$. Then the
   gauge
   \[
   \kappa(x)=\innerp{c}{x}+\indicator{\Kscr}(x)
   \]
   has the polar
   \begin{equation}
     \label{eq:kappa-polar-cone}
     \polar\kappa(u) = \inf\set{\alpha\ge0 | \alpha c\in\Kscr^{*}+u},
   \end{equation}
   with $\dom\kappa^\circ= \vspan\{c\} - \Kscr^*$.
   If $\Kscr$ is closed and $c\in \interior \Kscr^*$, then
   $\kappa$ has compact level sets, and $\dom\kappa^\circ=\Xscr$.
\end{proposition}
\begin{proof}
  From Proposition~\ref{sec2:prop1}, we have that
  \begin{align}
    \kappa^\circ(u)
      &=\sup\left\{\innerp{u}{x}\,\left|\,\kappa(x)\leq1\right.\right\}\notag\\
      &=\sup\left\{\innerp{u}{x}\,\left|\,\innerp{c}{x}\leq1\text{ and }x\in\Kscr\right.\right\}\label{eq:primal-prob11}\\
      &=\inf\left\{\alpha\geq0\,\left|\,\alpha c-u\in\Kscr^*\right.\right\},\notag
  \end{align}
  where the strong (Lagrangian) duality relationship in the last
  equality stems from the following argument. First consider the case
  where $u\in\dom\kappa^\circ$. Because the maximization problem
  in~\eqref{eq:primal-prob11} satisfies Slater's condition, equality
  follows from \citet[Corollary~28.2.2 and Theorem~28.4]{Roc70}.
  Next, consider the case where $u\notin\dom\kappa^\circ$, where
  $\kappa^{\circ}(u)=\infty$. The last equality then follows from weak
  duality. For the domain, note that the minimization problem is feasible
  if and only if $u\in\vspan\{c\}-\Kscr^{*}$; hence
  $\dom\kappa^\circ=\vspan\{c\}-\Kscr^*$.

  To prove compactness of the level sets of $\kappa$ when $\Kscr$ is
  closed and $c\in\interior\Kscr^*,$ define
  $\gamma:=\inf_x\set{\innerp{c}{x} | \norm{x}=1,\ x\in\Kscr}$ and
  observe that compactness of the feasible set in this minimization
  implies that the infimum is attained and that $\gamma>0.$ Thus, for
  any $x\in\Kscr\setminus\{0\}$, $\innerp c x\ge\gamma\norm{x}>0$ and,
  consequently, that
  $\set{x\in\Xscr|\kappa(x)\le\alpha}=\set{x\in\Kscr|\innerp c
    x\le\alpha}\subset\set{x\in\Xscr|\norm{x}\le\alpha/\gamma}$. This
  guarantees that the level sets of $\kappa$ are bounded, which
  establishes their compactness. From this and
  Proposition~\ref{sec2:prop1}(iii), we see that $\kappa^\circ(u)$ is
  finite for any $u\in \Xscr$.
\end{proof}

\begin{remark}\label{remark:polarsdp}
  Note that even though the polar gauge in~\eqref{eq:kappa-polar-cone}
  is closed, it is not necessarily the case that it has a closed
  domain. For example, let $\Kscr$ be the cone of PSD $2$-by-$2$
  matrices, and define
  \[
    c=\begin{pmatrix}1&0\\0&0\end{pmatrix}
    \textt{and}
    u_n=\begin{pmatrix}0&\phantom{-}1\\1&-\frac{1}{n}\end{pmatrix},
  \]
  for each $n=1,2,\ldots$. Use the
  expression~\eqref{eq:kappa-polar-cone} to obtain that
  $\kappa^\circ(u_n)=n$. Hence, $u_{n}\in\dom\kappa^{\circ}$, but
  $\lim_{n\to\infty}u_{n}\not\in\dom\kappa^{\circ}$.

  This is an example of a more general result described by
  \citet[Lemma 2.2]{ramana:1997}, which shows that the cone of PSD
  matrices is \emph{devious} (i.e., for every nontrivial proper face
  $F$ of $\Kscr$, $\vspan F+\Kscr$ is not closed). The concept of a
  devious cone seems to be intimately related to the closedness of the
  domain of polar gauges such as~\eqref{eq:kappa-polar-cone} because
  $\vspan\{c\}-\Kscr^*=-(\vspan F+\Kscr^*),$ where $F\subseteq\Kscr^*$
  is the smallest face of $\Kscr^*$ that contains $c$; see
  \cite[Proposition~3.2]{TunWolk:2012}.

  With that in mind, it is interesting to derive a representation for
  the closure of the domain of~\eqref{eq:kappa-polar-cone}.  It
  follows from \citet[Corollary 16.4.2]{Roc70} that
  $\cl\dom\kappa^\circ =\cl\left(\vspan\{c\}-\Kscr^*\right)
  =\big(\{c\}^\perp\cap\cl\Kscr\big)^\circ$.
\end{remark}

\subsubsection{Semidefinite conic gauge optimization}
\label{sec:SDP-gauge}

Here we give a concrete example of how to derive the gauge dual of a
semidefinite conic gauge optimization problem. Consider the feasible
semidefinite program
\begin{equation}\label{eq:non-neg-sdp}
\minimize{X}\quad\innerp{C}{X}\quad\st\quad \Ascr X=b,\ X\succeq0,
\end{equation}
where $C\succeq 0$, and $\Ascr:\Sscr^{n}\to\R^{m}$ is a linear map
from symmetric $n$-by-$n$ matrices to $m$-vectors. Define the gauge
objective $ \kappa(X) = \innerp{C}{X} + \delta_{\cdot\succeq0}(X)$,
set $\sigma=0$, and let $\rho=\norm{\cdot}$, i.e., the constraint set
is $\Cscr=\set{X|\Ascr X=b}$. Proposition~\ref{thm:gauge-plus-cone},
with $\Kscr$ equal to the (self-dual) PSD cone, gives the
gauge polar
\[
\begin{aligned}
  \polar\kappa(U)
  = \inf\set{\alpha\ge0|\alpha C\succeq U}.%
\end{aligned}
\]
The gauge dual~\eqref{g-dual2} then specializes to
\[
\minimize{\alpha,y}\quad\alpha
\quad\st\quad
\alpha\geq0,\ \innerp b y \ge 1,\ \alpha C\succeq\Ascr^*y,
\]
which is valid for all $C\succeq 0$.
This dual can be simplified by noting that
\[
\polar\kappa(U) = \max\{0,\ \lambda\submax\left(U,C\right)\},
\]
where $\lambda\submax\left(U,C\right)$ is the largest generalized
eigenvalue corresponding to the eigenvalue problem $Ux=\lambda
Cx$ (which might be $+\infty$ as in the example given in Remark~\ref{remark:polarsdp}).
Together with Corollary~\ref{cor:rho-axb-antipolar}, which gives
the antipolar of $\Cscr$,  and Theorem~\ref{thm:weak-duality}, which asserts
that the optimal dual value is positive, the gauge dual problem can then be written as
\[
\minimize{y}\quad\lambda\submax(\Ascr^{*}y,\ C)
\quad\st\quad
\innerp b y \ge 1.
\]
The lifted formulation of phase retrieval \citep{candes2012phaselift}
is an example of the
conic gauge optimization problem \eqref{eq:non-neg-sdp} with $C=I$. In
that case, \eqref{eq:non-neg-sdp} is the problem of minimizing the
trace of a PSD matrix that satisfies a set of
linear constraints.
The gauge
dual problem above is simply minimizing the maximum eigenvalue of
$\Ascr^{*}y$ over a single linear constraint.

\section{Discussion} \label{sec:discussion}

Our focus in this paper has been mainly on the duality aspects of
gauge optimization. The structure particular to gauge optimization
allows for an alternative to the usual Lagrange duality, which may be
useful for providing new avenues of exploration for modeling and
algorithm development. Depending on the particular application, it may
prove computationally convenient or more efficient to use existing
algorithms to solve the gauge dual rather than the Lagrange dual
problem. For example, some variation of the projected subgradient
method might be used to exploit the relative simplicity of the gauge
dual constraints in \eqref{g-dual2}.  As with methods that solve the
Lagrange dual problem, a procedure would be needed to recover the
primal solution. Although this is difficult to do in general, for
specific problems it is possible to develop a primal-from-dual
recovery procedure via the optimality conditions.

More generally, an important question left unanswered is if there
exists a class of algorithms that can leverage this special
structure. We are intrigued by the possibility of developing a
primal-dual algorithm specific to the primal-dual gauge pair.

The sensitivity analysis presented in \S\ref{sec:sensitivity-analysis} relied on
existing results from Lagrange duality. We would prefer, however, to
develop a line of analysis that is self-contained and based entirely
on gauge duality theory and some form of ``gauge multipliers''.
In this regard, if we define the value function as
$\tilde v(b,\sigma) = \inf_x\set{\kappa(x) + \delta_{\epi\rho}(b-Ax,\sigma)}$, then $\tilde v$
 is a gauge. It is
conceivable that sensitivity analysis could be carried out based on
studying its polar, given by
\begin{equation*}
    \tilde v^\circ(y,\tau) = \inf\set{\mu\ge 0| (y,\tau)\in \mu \Dscr}
    = \kappa^\circ(A^*y) + \delta_{\rho^\circ(\cdot)\le \cdot}(y,-\tau),
\end{equation*}
where $\Dscr = \set{(y,\tau)|
  \kappa^\circ(A^*y)\le 1, \rho^\circ(y)\le -\tau}$. This formula follows from Proposition~\ref{sec2:prop1}(iv)
and a computation of $\tilde v^*$ similar to the one leading to
\eqref{sec4:conjv_pre}. This approach would be in contrast to the usual
sensitivity analysis, which is based on studying a certain (convex)
value function and its conjugate.

\section*{Acknowledgments}
The authors are grateful to two anonymous referees for their
exceptionally careful reading of this paper, and for many helplful
suggestions and corrections. In particular, we are grateful to the
referees for generalizing a previous version of
Example~\ref{example:conic-gauge}

\bibliographystyle{my-plainnat}
\bibliography{friedlander,master,Pong}

\end{document}